\documentclass[12pt,a4paper]{article}

\usepackage{float}

\usepackage{tikz}

\usepackage{datetime}

\usepackage[all]{xy}

\usepackage{amssymb, amsmath, amsthm, eucal}

\usepackage{mathptmx}
\usepackage[scaled=.90]{helvet}
\usepackage{courier}

\usepackage[pdftex, hyperindex=true]{hyperref}

\hypersetup{
 colorlinks = true,
 linkcolor = black,
 urlcolor = blue,
 citecolor = black
}

\newcommand{\bbP}{\mathbb{P}}

\newcommand{\bbZ}{\mathbb{Z}}

\newcommand{\rmA}{\mathrm{A}}
\newcommand{\rmB}{\mathrm{B}}
\newcommand{\rmC}{\mathrm{C}}
\newcommand{\rmD}{\mathrm{D}}
\newcommand{\rmE}{\mathrm{E}}

\newcommand{\rmH}{\mathrm{H}}

\newcommand{\rmS}{\mathrm{S}}
\newcommand{\rmT}{\mathrm{T}}

\newcommand{\rmZ}{\mathrm{Z}}

\newcommand{\rmg}{\mathrm{g}}

\newcommand{\bfA}{\mathbf{A}}
\newcommand{\bfC}{\mathbf{C}}
\newcommand{\bfD}{\mathbf{D}}

\newcommand{\bfT}{\mathbf{T}}

\newcommand{\Mod}{\mathbf{Mod}}
\newcommand{\Per}{\mathbf{Per}}
\newcommand{\Dper}{\bfD^\text{\upshape per}}

\newcommand{\calT}{\mathcal{T}}
\newcommand{\calZ}{\mathcal{Z}}

\DeclareMathOperator{\CH}{CH}
\DeclareMathOperator{\HH}{HH}

\DeclareMathOperator{\id}{id}
\DeclareMathOperator{\op}{op}
\DeclareMathOperator{\gr}{gr}
\DeclareMathOperator{\dg}{dg}

\DeclareMathOperator{\Ho}{Ho}
\DeclareMathOperator{\Id}{Id}
\DeclareMathOperator{\Gr}{Gr}
\DeclareMathOperator{\Tot}{Tot}
\DeclareMathOperator{\Hom}{Hom}
\DeclareMathOperator{\Ext}{Ext}
\DeclareMathOperator{\Der}{Der}
\DeclareMathOperator{\Out}{Out}

\newtheorem{theorem}{Theorem}
\newtheorem{proposition}[theorem]{Proposition}
\newtheorem{corollary}[theorem]{Corollary}

\theoremstyle{definition}
\newtheorem{example}[theorem]{Example}

\newtheorem{remark}[theorem]{Remark}

\numberwithin{theorem}{section}
\numberwithin{equation}{section}
\numberwithin{figure}{section}

\title{Spectral sequences for Hochschild cohomology and graded centers of derived categories}

\author{Frank Neumann and Markus Szymik}

\date{March 2017}

\begin{document}

\maketitle

\renewcommand{\abstractname}{\vspace{-2\baselineskip}}

\begin{abstract}
\noindent 
The Hochschild cohomology of a differential graded algebra, or a differential graded category, admits a natural map to the graded center of its homology category:~the characteristic homomorphism. We interpret it as an edge homomorphism in a spectral sequence. This gives a conceptual explanation of the failure of the characteristic homomorphism to be injective or surjective, in general. To illustrate this, we discuss modules over the dual numbers, coherent sheaves over algebraic curves, as well as examples related to free loop spaces and string topology.

\vspace{\baselineskip}
\noindent MSC: 
16E40,	
18G40	
(14F05)	

\vspace{\baselineskip}
\noindent Keywords: Hochschild cohomology, center, spectral sequence, derived category
\end{abstract}


\section*{Introduction}

Differential graded categories seem to have been considered for the first time by Kelly~\cite{Kelly}, but it took a while before it was realized that these are much better behaved than the derived categories that are left over when the homological hatchet has done its work. Keller's ICM talk~\cite{Keller:ICM}, for instance, justifies this point of view.

Centers, and graded variants of this concept, have been investigated in various derived contexts, for instance for the derived categories of modules over~(commutative and non-commutative) algebras, derived categories of coherent sheaves in algebraic geo\-metry, as well as stable module categories in representation theory. See, for instance, Lowen--van den Bergh~\cite{Lowen+vandenBergh}, Avramov--Iyengar~\cite{Avramov+Iyengar}, Buchweitz--Flenner~\cite{Buchweitz+Flenner}, and Krause--Ye~\cite{Krause+Ye}. The non-linear and unstable situation in homotopy theory is addressed in~\cite{Dwyer+Szymik}. 

The purpose of this article is to shed light on the relationship between the centers of these derived categories on the one hand and the derived version of the center in the form of Hochschild cohomology on the other. It has already been known for some time that these can be related by the so-called characteristic homomorphism
\[
\HH_{\dg}^t(\bfA)\longrightarrow\rmZ^t_{\gr}(\rmH_\bullet\bfA)
\]
from the Hochschild cohomology to the center. See earlier work of Buchweitz--Flenner~\cite{Buchweitz+Flenner}, Lowen~\cite{Lowen}, Linckelmann~\cite[Sec.~2]{Linckelmann}, and Kuri\-ba\-ya\-shi~\cite[Sec.~5]{Kuribayashi}, for instance. Here we describe spectral sequences
\[
\HH_{\gr}^s(\rmH_\bullet\bfA\,;\,\Sigma^t\rmH_\bullet\bfA)\Longrightarrow\HH_{\dg}^{s+t}(\bfA)
\]
that have the characteristic homomorphisms as edge homomorphisms. See our Theorems~\ref{thm:main_ss} and~\ref{thm:identification} for precise statements. These spectral sequences can be used to get deeper insights into questions such as if a characteristic homomorphism is injective or surjective. A characteristic homomorphism is not injective if there is a term outside of the edge that survives the spectral sequence. A characteristic homomorphism is not surjective if there is a non-trivial differential involving classes on the edge; in that case one can say that the target of this differential is an obstruction to lifting the class from the graded center into Hochschild cohomology.

We will also describe spectral sequences of the form
\[
\rmH^p(\HH_{\gr}^q(\bfA))\Longrightarrow\HH_{\dg}^{p+q}(\bfA),
\]
see Theorem~\ref {thm:forgetful_ss} for the precise statement. Their edge homomorphisms take the form
\[
\rmH^p(\rmZ_{\gr}(\bfA))\longrightarrow\HH_{\dg}^p(\bfA).
\]
This is related to the homotopy limit problem in the sense of Thomason~\cite{Thomason} and Carlsson~\cite{Carlsson}, compare Remark~\ref{rem:hlp}.

Here is an outline of this article. In Section~\ref{sec:gdh} we reveal our conventions for gradings, differentials and homology, and in Section~\ref{sec:cat} we do so for the various sorts of categories that we will be referring to. In Section~\ref{sec:auto} we explain the relation between gradings and automorphisms; this allows us to give slightly more general definitions than those published so far. In Sections~\ref{sec:Z} and~\ref{sec:HM} we set up the centers and the Hochschild cohomology of these categories, respectively. In the main Section~\ref{sec:sss} we present and discuss the spectral sequences that interpolate between these invariants, and the final Sections~\ref{sec:alg},~\ref{sec:geo}, and~\ref{sec:top} contain examples from algebra, algebraic geometry, and algebraic topology, respectively, to illustrate the complexity of the situation.


\section{Gradings, differentials, and homology}\label{sec:gdh}

Throughout this text, we fix a ground field~$\mathfrak{K}$, and we will suppress it from the notation whenever convenient. For instance, if~$V$ and~$W$ are two vector spaces, then their tensor product~\hbox{$V\otimes W$} and their homomorphism space~$\Hom(V,W)$ will always be taken over the ground field. This makes the category of vector spaces a closed symmetric monoidal category with respect to the usual symmetry~$v\otimes w\mapsto w\otimes v$.



A {\em graded vector space}~$V$ is just a family~$(V_n\,|\,n\in\bbZ)$ of vector spaces, indexed by the abelian group~$\bbZ$ of integers.
Every vector space~$V$ gives rise to a graded vector space that is concentrated in degree~$0$. We recover the vector space from this graded vector space as the degree~$0$ part. We will use this observation to assume that everything is graded throughout this text, with apparently ungraded objects embedded in degree~$0$ if necessary.
If~$V$ and~$W$ are two graded vector spaces, their graded tensor product~\hbox{$V\otimes W$} is defined by the unsurprising formula
\[
(V\otimes W)_n=\bigoplus_{i+j=n}V_i\otimes W_j,
\]
and their graded homomorphism vector space~$\Hom(V,W)$ is defined by
\[
\Hom(V,W)_n=\prod_{m}\Hom(V_{m-n},W_m)=\prod_{m}\Hom(V_m,W_{m+n}),
\]
so that~$\Hom$ is right adjoint to the tensor product~$\otimes$, as it should be. This endows the category of graded vector spaces with the structure of a closed symmetric monoidal category with respect to the Koszul symmetry~\hbox{$v\otimes w\mapsto (-1)^{|v||w|}w\otimes v$}, where the superscripts~$|v|$ and~$|w|$ refer to the grades in which the vectors~$v$ and~$w$ live.
We will sometimes have reasons to shift graded vector spaces in either direction of the index. If~$V$ is a graded vector space, and~$n$ is an integer, then the graded vector spaces~$\Omega^n V$ and~$\Sigma^n V$ are defined by~$(\Omega^n V)_m = V_{m+n}$ and~$(\Sigma^n V)_m = V_{m-n}$.
This actually defines an automorphisms~$\Sigma$ of the category of graded vector spaces with inverse~$\Omega$. Finally, we will sometimes write~$V^t=V_{-t}$
whenever this is more convenient.



A {\em differential graded vector space} is a graded vector space~$X$ together with a {\em differential}: an endomorphism~$\delta_X\in\Hom(X,X)_{-1}$ of degree~$-1$ such that~$\delta_X^2=0$.
In other words, the differential graded vector spaces are just the chain complexes. 
Every graded vector space can be regarded as a differential graded vector space that has zero differential. 
If~$X$ and~$Y$ are two differential graded vector spaces, their tensor product~\hbox{$X\otimes Y$} has differential~\hbox{$\delta(x\otimes y)=\delta x\otimes y+(-1)^{|x|}x\otimes\delta y$} and the homomorphism space~$\Hom(X,Y)$ has differential (still of degree~$-1$) determined by~\hbox{$\delta(fx)=(\delta f)x+(-1)^{|f|}f(\delta x)$.} This makes the category of differential graded vector spaces a closed symmetric monoidal category with respect to the same Koszul symmetry as for graded vector spaces. Both the shift automorphism~$\Sigma$ and its inverse~$\Omega$ change the sign of the differential. 



There are at least two useful ways to pass from a differential graded vector space to a graded vector space: On the one hand, we can forget about the differential, and on the other hand, we can pass to {\it homology}. 
The homology~$\rmH_\bullet X$ of a differential graded vector space~$X$ is a graded vector space. We write~\hbox{$\rmH_nX=(\rmH_\bullet X)_n$} for the term in degree~$n$, as usual. 
Since we are working over a field, we have very simple formulas for the homology of tensor products and homomorphisms spaces:
\[
\rmH_\bullet(X\otimes Y)\cong(\rmH_\bullet X)\otimes(\rmH_\bullet Y)
\]
and
\[
\rmH_\bullet\Hom(X,Y)\cong\Hom(\rmH_\bullet X,\rmH_\bullet Y),
\]
both as graded vector spaces. Here, the second equation deserves to be spelled out for the particular grades:
\[
\rmH_n(\Hom(X,Y))\cong\Hom(\rmH_\bullet X,\Omega^n\rmH_\bullet Y)=\Hom(\Sigma^n\rmH_\bullet X,\rmH_\bullet Y).
\]


\section{Differential graded categories}\label{sec:cat}

Throughout this paper, we will have occasion to deal with differential graded categories, which are our main interest, and the more basic graded categories, which can be thought of as differential graded categories without differentials. Even more basic are the linear categories, which can be thought of as graded categories concentrated in degree~$0$. This section gives precise definitions and introduces our notation alongside with them. Compare with~\cite{Keller:ENS}, where the conventions are different, however.

A {\em differential graded category}~$\bfA$ consists of a set of objects together with differential graded vector spaces~$\bfA(x,y)$ for each pair~$(x,y)$ of objects, composition homomorphisms
\begin{equation}\label{eq:composition}
\bfA(y,z)\otimes\bfA(x,y)\longrightarrow\bfA(x,z),\,g\otimes f\longmapsto gf
\end{equation}
that are associative, and identities~$\id_x\in\bfA(x,x)_0$ for each object~$x$ that are neutral with respect to these compositions. A differential graded category with one object~$\star$ is essentially just a differential graded algebra~$C=\bfA(\star,\star)$. If~$C$ is a differential graded algebra, then the category of differential graded~$C$--modules~(with a suitable size restriction) is a differential graded category. 


A {\em graded category} is just a differential graded category~$\bfC$ where the differential is zero. A graded category~$\bfC$ with one object~$\star$ is essentially the same structure as a graded algebra~$A=\bfC(\star,\star)$. For instance, if~$\mathfrak{K}$ denotes our ground field, the algebra~$A$ could be~$\mathfrak{K}[\delta]/\delta^2$ with~$\delta$ in dimension~$-1$. Then the graded~$A$--modules are precisely the differential graded vector spaces. If~$A$ is a graded algebra, then the category of graded~$A$--modules~(with a suitable size restriction) is a graded category. 

%


A {\em linear category}~$\bfC$ is a graded category~$\bfC$ that is concentrated in degree~$0$. An algebra~$A$ can be thought of as a linear category with exactly one object~$\star$ and endomorphism algebra~$\bfC(\star,\star)=A$. Then the composition~\eqref{eq:composition} is just the multiplication in the algebra. If~$A$ is an algebra, then the category~$\Mod_A$ of~$A$--modules~(with a suitable size restriction) is a linear category. 



As already mentioned before, there are two ways to pass from a differential graded vector space to a graded vector space, and there are, correspondingly, two ways to pass from a differential graded category to a graded category: On the one hand, we can forget the differentials, and on the other hand we can pass to homology. The latter path leads to the (graded) homology categories that we will briefly recall here, compare~\cite[Sec.~3]{Keller:ENS}. 

Let~$\bfA$ be a differential graded category. The {\em homology category}~$\rmH_\bullet\bfA$ of~$\bfA$ is the graded category that has the same objects, and its morphism spaces are defined by~\hbox{$(\rmH_\bullet\bfA)(x,y)=\rmH_\bullet(\bfA(x,y))$}. The homology category of a differential graded category is a graded category. 
If~$C$ is a differential graded algebra, then~$\rmH_\bullet C$ is its graded homology algebra. This type of example will be studied in the final Section~\ref{sec:top}, where~\hbox{$C=\rmC X$} will be the differential graded algebra of cochains on a space~$X$.

A complex of modules over an algebra~$A$ is {\em perfect} if it is (quasi-isomorphic to) a bounded complex of finitely generated projective~$A$--modules~\cite{SGA6}. If~$\Per_A$ denotes the differential graded category of perfect complexes over an algebra~$A$, then~\hbox{$\rmH_\bullet\Per_A=\Dper_A$} is the {\em perfect derived category} of~$A$ with its usual grading from its triangulated structure. Similarly, if~$\Per_X$ is the differential graded category of perfect complexes over a scheme~$X$, then~\hbox{$\rmH_\bullet\Per_X=\Dper_X$} is the perfect derived category of~$X$. These types of examples will be studied in Sections~\ref{sec:alg} and~\ref{sec:geo}, respectively.


\section{Automorphisms}\label{sec:auto}

So far, automorphisms of the categories under consideration have not been mentioned, except for the suspension and de-suspension functors~$\Sigma$ and~$\Omega$. Since automorphisms will play a major role when it comes to defining suitable notions of centers and Hochschild cohomology later on, we will now clarify their place in our context. 

From a conceptual perspective, an automorphism is essentially the same thing as an action of the additive group of the integers~$\bbZ$, and one may be tempted to develop a theory in the generality of group actions on categories. We resist for the benefit of readability. The interested reader may find further inspiration in~\cite{HPS},~\cite{ENO}, and~\cite{TV}, for instance.



In many cases, the (differential) graded categories with graded automorphisms that we will be concerned with are the result of a simple procedure that produces graded categories with graded automorphisms from categories with automorphisms. This procedure will now be explained.

Let~$T\colon\bfC\to\bfC$ be a linear automorphism of a linear category~$\bfC$. This yields a graded category with the same objects and
\begin{equation}\label{eq:defines}
	\bfC(x,y)_n=\bfC(T^nx,y)
\end{equation}
for all objects~$x$ and~$y$. The displayed formula shows that the category~$\bfC$ we started with can be recovered as the degree~$0$ part of this graded category. The automorphism~$T$ extends to a graded automorphism of the graded category,
\begin{equation}\label{eq:preserves}
	\bfC(x,y)_n=\bfC(T^nx,y)\overset{T}{\longrightarrow}\bfC(T^{n+1}x,Ty)=\bfC(Tx,Ty)_n.
\end{equation}

It is obvious from~\eqref{eq:defines} that the identity functor on a category~$\bfC$ defines a rather uninteresting graded category. In contrast, if~$(\bfT,\Delta,\Sigma)$ is a triangulated category with shift automorphism~$\Sigma$, then this leads to the usual graded category with degree~$0$ part~$\bfT$. Note that the class~$\Delta$ of distinguished triangles is not needed anywhere here.

On the other hand, given any graded category~$\bfC$, its identity functor is a graded automorphism, but it does not have to come from a functor on the degree~$0$ part in the way just described. In fact, rather many interesting graded categories are not defined by automorphisms of its degree~$0$ part. This can be the case, for instance, if~$\bfC$ is a graded algebra~\hbox{$A=\bfC(\star,\star)$}, where there is only one object~$\star$. In that case, an automorphism is just a graded automorphism of the graded algebra.



It seems natural to think that every category~$\bfC$ has a preferred automorphism: the identity functor~$\Id_\bfC$. This is certainly true when~$\bfC$ is concentrated in degree~$0$. But, in the graded context, this is no longer the most useful thing to do. Here is a more useful alternative:
Let~$\bfC$ be a graded category. The {\it parity functor}
\begin{equation}\label{eq:parity}
\Gr_\bfC\colon\bfC\longrightarrow\bfC
\end{equation}
is defined as the identity on objects, and~$\Gr_\bfC\colon\bfC(x,y)\to\bfC(x,y)$ sends a morphism~$f$ to~$(-1)^{|f|}f$.


%


\section{Centers}\label{sec:Z}

The zeroth Hochschild cohomology of an algebra is its center. We therefore need to review the appropriate notions of centers for the various types of categories at hand. If~$\bfC$ is just a linear category, that is a graded category concentrated in degree~$0$, then its {\it center}~$\rmZ(\bfC)$ is the set of natural transformations~\hbox{$\Id_\bfC\to\Id_\bfC$} from the identity functor to itself. (This is sometimes called the {\em Bernstein center} of~$\bfC$, in reference to Bernstein's work~\cite{Bernstein} which discusses a particular case of interest.) The center is a commutative algebra under composition. If~$A$ is an algebra, thought of as a linear category with only one object, then we recover the usual concept of the center~$\rmZ(A)$ of~$A$. If~$\bfC=\Mod_A$ is a category of~$A$--modules that contains a free~$A$--module of rank~$1$, then the center~$\rmZ(\Mod_A)$ of that category is isomorphic to the center~$\rmZ(A)$ of the algebra, by evaluation on the generator.



Let now~$\bfC$ be a graded category. There is only one possible definition of its center: The {\em graded center}~$\rmZ_{\gr}^\bullet(\bfC)$ is the graded vector space that is defined as follows. The elements~$\Phi\in\rmZ_{\gr}^n(\bfC)$ in degree~$n$ are the families~$(\,\Phi_x\,|\,x\in\bfC\,)$ of morphisms~\hbox{$\Phi_x\in\bfC(x,x)_n$} that are natural in the sense that
\begin{equation}\label{eq:natural}
f\Phi_x=(-1)^{|f|n}\Phi_yf
\end{equation}
for all morphisms~$f\colon x\to y$. The graded center~$\rmZ_{\gr}^\bullet(\bfC)$ of a graded category is a graded algebra~(with respect to composition of morphisms) which is commutative in the graded sense. For any graded category~$\bfC$, the degree~$0$ part~$\rmZ_{\gr}^0(\bfC)$ of the graded center~$\rmZ_{\gr}^\bullet(\bfC)$ of~$\bfC$ is contained in the center~$\rmZ(\bfC_0)$ of the degree~$0$ part~$\bfC_0$ of~$\bfC$. But there is no reason why these two should be equal if~$\bfC$ is not concentrated in degree~$0$.



We can now present a definition of the graded center of a graded category together with a graded automorphisms~(in the sense of the previous Section~\ref{sec:auto}).
Let~$(\bfC,T)$ be a graded category together with an automorphism that preserves the grading. Then the {\em graded center}~\hbox{$\rmZ_{\gr}^\bullet(\bfC,T)\subseteq\rmZ_{\gr}^\bullet(\bfC)$} is the graded vector space that is defined as follows. The vector space~$\rmZ_{\gr}^n(\bfC,T)$ is the subspace of the space~$\rmZ_{\gr}^n(\bfC)$ that consists of the elements which are compatible with~$T$ up to sign: families~\hbox{$(\,\Phi_x\,|\,x\in\bfC\,)$} of morphisms~\hbox{$\Phi_x\in\bfC(x,x)_n$} that are natural~\eqref{eq:natural} and such that
\begin{equation}\label{eq:sign_1}
\Phi_{T x}=(-1)^nT\Phi_x
\end{equation}
for all objects~$x$.
The graded center~$\rmZ_{\gr}^\bullet(\bfC,T)$ of a graded category together with an automorphism that preserves the grading is a graded algebra (with respect to composition of morphisms). It is commutative in the graded sense, because it is a subalgebra of the algebra~$\rmZ_{\gr}^\bullet(\bfC)$ that is commutative in the graded sense.

One may wonder if, given a graded category~$\bfC$, there is a canonical choice of an automorphism~$T$ of~$\bfC$ such that~$\rmZ_{\gr}^\bullet(\bfC,T)$ is the plain center~$\rmZ_{\gr}^\bullet(\bfC)$. There is no reason why the identify functor~$T=\Id_{\bfC}$ should work, due to the sign in~\eqref{eq:sign_1}. In contrast, the parity functor~$\Gr_\bfC$ always works, so that we have~\hbox{$\rmZ_{\gr}^\bullet(\bfC,\Gr_\bfC)=\rmZ_{\gr}^\bullet(\bfC)$}, because the signs cancel each other.



In the special case where the graded category and its graded automorphism come from an automorphism on a category, as in Section~\ref{sec:auto}, we have the following result that is an immediate consequence of the definitions.

\vbox{\begin{proposition}\label{prop:agrees}
Given a category~$\bfC$ together with an automorphism~$T$, the degree~$n$ part of the graded center of the associated graded category with graded automorphism is the same as the subspace of the space of natural transformations~$\Id_\bfC\to T^n$ of functors~\hbox{$\bfC\to\bfC$} such that~\eqref{eq:sign_1} holds.
\end{proposition}}

To lighten the notation, we will simply write~$\rmZ_{\gr}^\bullet(\bfC,T)$ for the graded center of the associated graded category with graded automorphism. This cannot lead to confusion, because~$\rmZ_{\gr}^\bullet(\bfC,T)$ is undefined so far if~$\bfC$ is only a linear category without grading. 

Let~$(\bfT,\Delta,\Sigma)$ be a triangulated category with shift automorphism~$\Sigma$. Then Proposition~\ref{prop:agrees} implies that the graded center of the associated graded category with graded automorphism is the graded center of the triangulated category~$\bfT$ as studied by Buchweitz and Flenner~\cite{Buchweitz+Flenner}, Linckelmann~\cite{Linckelmann}, and Krause and Ye~\cite{Krause+Ye}. It seems from this discussion that our definition is slightly more general than theirs.



If~$\bfA$ is a differential graded category, then we can of course ignore the differential. This gives a graded category~$\bfA$, and we are free to consider the graded center~$\rmZ_{\gr}^\bullet(\bfA)$. This is certainly a graded algebra. Remembering the differential on~$\bfA$, we see that the graded center~$\rmZ_{\gr}^\bullet(\bfA)$ has the structure of a differential graded algebra. Passing to homology we get the graded algebra~$\rmH_\bullet(\rmZ_{\gr}^\bullet(\bfA))$. We will soon see that this graded algebra re-appears when we discuss the edge homomorphisms of the forgetful spectral sequences in Section~\ref{sec:forgetful_ss}. 

\begin{remark}\label{rem:not_defined}
At this point it may be tempting to define the ``differential graded center'' of the differential graded category~$\bfA$ either as the differential graded algebra~$\rmZ_{\gr}^\bullet(\bfA)$ or perhaps rather as its graded homology algebra~$\rmH_\bullet(\rmZ_{\gr}^\bullet(\bfA))$ from above. But, we will see later that, from a conceptual point of view, neither of these is entirely adequate, compare Remark~\ref{rem:justification}.
\end{remark} 

The discussion here would have been the same in the presence of an automorphism; only the notation would have been a little bit more involved.



We close this section with an example that illustrates the role of the automorphism when it comes to calculating centers: We will describe the graded centers of the graded categories~$\Per_A$ of perfect complexes over algebras~$A$ that are differential graded algebras concentrated in degree~$0$. The automorphism~$\Sigma$ is given by the shift functor.

\begin{proposition}\label{prop:shift}
There are canonical isomorphisms
\[
\rmZ^0_{\gr}(\Per_A,\Sigma)\cong\rmZ^1_{\gr}(\Per_A,\Sigma)\cong\rmZ(A),
\] 
and the graded center is trivial in all other degrees~$t\not=0,1$.
\end{proposition}

\begin{proof} 
Let us first study the part in degree~$0$. We use the chain complexes~$\rmD^n_A$ with~$(\rmD^n_A)_k=A$ for~$k=n,n-1$ and differential~$\id_A$ between these groups; everything else is zero. Then there is a natural isomorphism~$\Hom(\rmD^n_A,V)\cong V_n$, so that~$\rmD^n_A$ represents the functor~\hbox{$V\mapsto V_n$}. It follows that any~$\Phi$ in the center is determined by a family~$(\,a_n\,|\,n\in\bbZ\,)$ of elements~$a_n$ in~$A$: The chain map~$\Phi_V$ sends~$v$ to~$a_nv$ if~$v$ is an element of degree~$n$. Evaluating this on~$V=\rmD^n_A$ shows that~$a_n=a_{n-1}$ for all~$n$, so that~$\Phi$ is determined by an element~$a$ in~$A$. Naturality forces~$a$ to lie in the center~$\rmZ(A)$ of the algebra~$A$. Conversely, it is obvious that multiplication by an element~$a$ in the center~$\rmZ(A)$ of the algebra~$A$ defines an element in the center of the category. Note that these elements are automatically compatible with the shift automorphism.

For the other degrees~$t\not=0$, we again use the chain complexes~$\rmD^n_A$. Since these are concentrated in degrees~$n$ and~$n-1$ of distance~$1$, the last statement immediately follows. For the second one, we see that we necessarily have elements~$a_n\in\rmZ(A)$ such that~$\Phi$ is given by ~$v\mapsto a_n\delta v$ for all elements~$v$ of degree~$n$. Here, since we have~$\delta^2=0$, there are no relations between the elements~$a_n$. This shows
\[
\rmZ^1_{\gr}(\Per_A)\cong\prod_{n\in\bbZ}\rmZ(A).
\] 
Since we are imposing shift invariance, this then implies that one of the elements~$a_n$ is enough to determine the others.
\end{proof}

Note that we have~$[a_n\delta v]=[\delta a_nv]=0$ in homology, so that the elements in the degree~$1$ part~$\rmZ^1_{\gr}(\Per_A)$ of the center all act trivially on the derived category.



\section{Hochschild cohomology}\label{sec:HM}

In this section we recall the definition of Hochschild cohomology for graded categories and differential graded categories, possibly with automorphisms.

Let~$\bfC$  be a (small) graded category. Let~$M\colon\bfC^{\op}\otimes\bfC\to\Mod_\mathfrak{K}$ be a graded linear functor to the category of vector spaces over the ground field~$\mathfrak{K}$, also known as a~{\em~$(\bfC,\bfC)$--bimodule}. Consider the cosimplicial vector space whose~$s$--th term is the vector space
\[
\CH_{\gr}^s(\bfC\,;\,M)=\prod_{x_0,\dots,x_s}\Hom( \bfC(x_1,x_0)\otimes\dots\otimes\bfC(x_s,x_{s-1}) ,M(x_s,x_0))
\]
where the indices~$x_0,\dots,x_s$ range over the~$(s+1)$--tuples of objects of~$\bfC$. 
The cosimplicial structure maps are the usual ones.
From this cosimplicial vector space, the alternating sum formula produces a cochain complex~$\CH_{gr}(\bfC\,;\, M)$ in graded vector spaces, the {\em Hochschild cochain complex}. Compare~\cite[Sec.~17]{Mitchell} and also~\cite{Baues+Wirsching}. The {\em Hochschild cohomology} of the graded category~$\bfC$ is the cohomology 
\[
\HH_{\gr}^s(\bfC\,;\,M)=\rmH^s(\CH_{\gr}(\bfC\,;\,M))
\]
of the cochain complex~$\CH_{\gr}(\bfC\,;\,M)$. The Hochschild cohomology of a graded category is a graded commutative algebra.

\begin{example}\label{ex:der_nat_trf}
Let~$F,G\colon\bfC\to\bfC'$ be two graded linear functors between graded categories. Then~$M(x_s,x_0)=\bfC'(Fx_s,Gx_0)$ gives a~$(\bfC,\bfC)$--bimodule~$M=\bfC'(F,G)$, and here the cohomology~$\HH_{gr}^\bullet(\bfC\,;\,M)=\HH_{gr}^\bullet(\bfC\,;\,\bfC'(F,G))$ should be thought of as the graded vector space of derived natural transformations~$F\to G$. The special case~\hbox{$F=\Id_\bfC=G$} leads to the bimodule~$\bfC=\bfC(\Id_\bfC,\Id_\bfC)$, and this case recovers the (plain) Hochschild cohomology~$\HH_{gr}^\bullet(\bfC)$ of the graded category~$\bfC$. 
\end{example}



Note that by this definition of the Hochschild cohomology~$\HH_{\gr}^s(\bfC)$ of a graded category~$\bfC$, the result is a graded vector space, so that it makes sense to write~$\HH_{\gr}^s(\bfC)_t$, or~$\HH_{\gr}^{s,-t}(\bfC)$ as will often be convenient. Using the suspension~$\Sigma$ on the graded morphism spaces, we can actually do without this second grading by using coefficients:~\hbox{$\HH_{\gr}^{s,t}(\bfC\,;\,M)=\HH_{\gr}^{s,0}(\bfC\,;\,\Sigma^tM)$}. We will abbreviate this as~$\HH_{\gr}^s(\bfC\,;\,\Sigma^tM)$. In this notation (and the one introduced in Example~\ref{ex:der_nat_trf}) we can write:

\begin{proposition}\label{prop:identification_1}
We have~$\HH_{\gr}^0(\bfC\,;\,\Sigma^t\bfC)=\rmZ_{\gr}^t(\bfC)$ if~$\bfC=\bfC(\Id_\bfC,\Id_\bfC)$.
\end{proposition}

\begin{proof}
This follows immediately from the definition of the first differential in the Hochschild complex.
\end{proof}



If~$T$ is an automorphism of the graded category~$\bfC$, then we can generalize the process that leads to the definition of the graded center, and introduce the subspaces~\hbox{$\CH_{\gr}^s(\bfC,T)\subseteq\CH_{\gr}^s(\bfC)$} by requiring that the cochains are~$T$--invariant up to the natural sign~\eqref{eq:sign_1}:
\begin{equation}\label{eq:sign_2}
\Phi(T f_1,\dots,T f_s)=(-1)^{|\Phi|}T(\Phi(f_1,\dots,f_s)).
\end{equation}
These subspaces are preserved by the differential, so that we get induced homomorphisms~\hbox{$\HH_{\gr}^s(\bfC,T)\to\HH_{\gr}^s(\bfC)$} between the Hochschild cohomology groups.

\begin{remark}\label{rem:TGr}
When the automorphism~$T$ is the parity functor~$\Gr_\bfC$ of~$\bfC$ as in~\eqref{eq:parity}, then we have~\hbox{$\HH_{\gr}^s(\bfC,\Gr_\bfC)=\HH_{\gr}^s(\bfC)$},
because for the signs in~\eqref{eq:sign_2} we get
\[
|\Phi(T f_1,\dots,T f_s)|=|f_1|+\dots+|f_s|
\]
and
\[
|T(\Phi(f_1,\dots,f_s))|=|\Phi|+|f_1|+\dots+|f_s|,
\]
respectively.
\end{remark}

The analogue of Proposition~\ref{prop:identification_1} is:

\begin{proposition}\label{prop:identification_2}
We have~$\HH_{\gr}^{0,t}(\bfC,T)\cong\rmZ_{\gr}^t(\bfC,T)$ as graded algebras, for every graded category~$\bfC$ with automorphism~$T$.
\end{proposition}



Let now~$\bfA$ be a (small) differential graded category, and~$M$ an~$(\bfA,\bfA)$--bimodule. Consider the cosimplicial differential graded vector space whose~$s$--th term is the differential graded vector space~$\rmE_0^{s}(\bfA\,;\,M)$ given by
\[
\rmE_0^{s,t}(\bfA\,;\,M)=\prod_{x_0,\dots,x_s}\Hom( \bfA(x_1,x_0)\otimes\dots\otimes\bfA(x_s,x_{s-1}) , M(x_s,x_0) )^t
\]
where the indices~$x_0,\dots,x_s$ range over the objects of~$\bfA$. The cosimplicial structure maps are given again by the usual formulas. The alternating sum produces a cochain complex in differential graded vector spaces from this cosimplicial differential graded vector space.

\begin{remark}\label{rem:not_wanted}
In contrast to the procedure for graded categories, we will {\em not} consider this as the cohomological Hochschild complex, and its cohomology is {\em not} the Hochschild cohomology. We will explain later~(in Section~\ref{sec:forgetful_ss}) how these too naive ideas lead to cohomology groups that contain more than what is usually considered as interesting; what needs to be cut down can be made precise, and we will do so later~(again, in Section~\ref{sec:forgetful_ss}). Before that, we recall the usual~(genuinely interesting) definition. 
\end{remark}

The definition of the Hochschild cohomology of a differential graded category starts from the observation that a cochain complex in differential graded vector spaces is a bicomplex: The {\em Hochschild complex} of a differential graded category~$\bfA$ is the~(product) total complex~$\CH_{\dg}(\bfA\,;\,M)=\Tot(\rmE_0(\bfA\,;\,M))$ of the bicomplex~$\rmE_0(\bfA\,;\,M)$. This is a cochain complex in vector spaces. Its cohomology~\hbox{$\HH_{\dg}^s(\bfA\,;\,M)=\rmH^s(\CH_{\dg}(\bfA\,;\,M))$} is the {\em Hochschild cohomology} of~$\bfA$ as a differential graded category. This is a standard definition, see~\cite[5.4]{Keller:ICM}, for instance.

\begin{remark}\label{rem:justification}
Proposition~\ref{prop:identification_1} might suggest that the zeroth Hochschild cohomology of a differential graded category is the ``differential graded center'' of that differential graded category. However, as it stands, this does not have a meaning, because the latter is not defined! (Compare with Remark~\ref{rem:not_defined}.) Instead, the Hochschild cohomology of a differential graded category can be related to the graded center of the derived category by means of the characteristic homomorphism. This will be explained in Section~\ref{sec:char_ss}.
\end{remark}



If we are given a graded automorphism~$T$ of the differential graded category~$\bfA$, then we can again introduce subspaces~\hbox{$\rmE_0(\bfA,T)\subseteq\rmE_0(\bfA)$} by requiring that the cochains are~$T$--invariant, up to the natural sign as before in~\eqref{eq:sign_2}. This defines a subcomplex~\hbox{$\CH_{\dg}(\bfA,T)=\Tot(\rmE_0(\bfA,T))$} of~$\Tot(\rmE_0(\bfA))=\CH_{\dg}(\bfA)$, so that we get induced homomorphisms~\hbox{$\HH_{\dg}^s(\bfA,T)\to\HH_{\dg}^s(\bfA)$} between the Hochschild cohomology groups. When the automorphism~$T$ is the parity functor~$\Gr_\bfC$ of~$\bfC$ as in~\eqref{eq:parity}, then we have~\hbox{$\HH_{\dg}^s(\bfA,\Gr_\bfA)=\HH_{\dg}^s(\bfA)$} 
for the same reason as given in Remark~\ref{rem:TGr}.


\section{The spectral sequences}\label{sec:sss}

Since the Hochschild cohomology of a differential graded category~$\bfA$ with coefficients in an~$(\bfA,\bfA)$--bimodule~$M$ is defined as the cohomology of the~(product) totalization of the double complex~$\rmE_0(\bfA\,;\,M)$ given by
\begin{equation}\label{eq:double}
\rmE_0^{s,t}(\bfA\,;\,M)=\prod_{x_0,\dots,x_s}\Hom( \bfA(x_1,x_0)\otimes\dots\otimes\bfA(x_s,x_{s-1}), M(x_s,x_0) )^t,
\end{equation}
there are automatically two spectral sequences that calculate it. See~\cite[5.6]{Weibel}, for instance, or any other text on homological algebra. These will be discussed in this section in general terms. Concrete examples will be presented in the following sections.


\subsection{The characteristic spectral sequence}\label{sec:char_ss}

The characteristic spectral sequence is obtained from the double complex~\eqref{eq:double} by calculating the internal (i.e.~with respect to~$t$) differentials first. Since homology commutes with products and we are working over a field, that gives
\[
\rmE_1^{s,t}(\bfA\,;\,M)=\prod_{x_0,\dots,x_s}\Hom(
\rmH_\bullet(\bfA(x_1,x_0))\otimes\dots\otimes\rmH_\bullet(\bfA(x_s,x_{s-1})),\rmH_\bullet(M(x_s,x_0)) 
)^t.
\]
This~$\rmE_1$ page carries a (Hochschild) differential in the~$s$--direction, and its cohomology is the Hochschild cohomology of the graded category~$\rmH_\bullet\bfC$ with coefficients in the bimodule~$\rmH_\bullet M$. So we get:

\begin{theorem}\label{thm:main_ss}
For every differential graded category~$\bfA$ and every~$(\bfA,\bfA)$--bi\-module~$M$ there exists a spectral sequence with~\hbox{$\rmE_2^{s,t}\cong\HH_{\gr}^s(\rmH_\bullet\bfA\,;\,\Sigma^t\rmH_\bullet M)$} and that converges to the Hochschild cohomology~$\HH_{\dg}^{s+t}(\bfA\,;\,M)$ of the differential graded category~$\bfA$ with coefficients in~$M$.
\end{theorem}

The spectral sequence in Theorem~\ref{thm:main_ss} will be referred to as the {\it characteristic spectral sequence}. The characteristic spectral sequence is a right-plane spectral sequence of cohomological type, so that~$\rmE_r^{s,t}\to\rmE_r^{s+r,t-r+1}$ is the differential on the~$r$--th page.

The spectral sequence in Theorem~\ref{thm:main_ss} is the differential graded analogue of the spectral sequence derived in~\cite{Szymik} in the (non-linear and unstable) context of simplicial or topological categories, i.e.~categories that have mapping {\em spaces} of morphisms. For any such category~$\bfC$ there is a homotopy coherent center~$\calZ(\bfC)$ which is a space. The (Bousfield--Kan type) spectral sequence in~\cite{Szymik} computes the homotopy groups~$\pi_\bullet\calZ(\bfC)$, and the initial term~$\rmE^2_{0,0}$ receives a map from the center~(in the ordinary sense) of the homotopy category~$\Ho(\bfC)$. The associated edge homomorphism~$\rmZ(\Ho\bfC)\to\pi_0\calZ(\bfC)$ is discussed in~\cite{Szymik} as well. The edge homomorphism that comes with the spectral sequence in Theorem~\ref{thm:main_ss} will be described in the following Section~\ref{sec:char_ss}. Bicategories can be thought of as (very) special simplicial categories, and~\cite{Meir+Szymik} contains a description of the resulting situation in detail.


Let us spell out the interesting one-object-case of Theorem~\ref{thm:main_ss}:

\vbox{\begin{corollary}\label{cor:dgass}
For every differential graded algebra~$C$ and bimodule~$M$ there is a characteristic spectral sequence with~\hbox{$\rmE_2^{s,t}\cong\HH_{\gr}^s(\rmH_\bullet C\,;\,\Sigma^t\rmH_\bullet M)$} that converges to the Hochschild cohomology~$\HH_{\dg}^{s+t}(C\,;\,M)$ of the differential graded algebra~$C$ with coefficients in~$M$.
\end{corollary}}

In the special case of differential graded algebras, the existence of spectral sequences that compute Hochschild homology is also asserted in~\cite[5.3.4]{Loday}. Another predecessor of the spectral sequence in Corollary~\ref{cor:dgass} originates in algebraic topology, this time in {\em stable} homotopy theory: Differential graded algebras~$C$ naturally give rise to~$\rmA_\infty$ ring spectra~$\rmH C$ over the integral Eilenberg--Mac Lane spectrum~$\rmH\bbZ$, and conversely, in such a way that the homology~$\rmH_\bullet C$ of a differential graded algebra is naturally isomorphic to the homotopy~$\pi_\bullet\rmH C$ of the associated~$\rmA_\infty$ ring spectrum~$\rmH C$. Shipley~\cite{Shipley} actually made more refined statements, but the present summary conveys the idea. In a context closely related to~$\rmA_\infty$ ring spectra, B\"okstedt~(unpublished) has established spectral sequences that compute the~(topological) Hochschild~(co)homology from the~(ordinary) Hochschild~(co)homology of its homotopy ring. See~\cite[Ch.~IX]{EKMM} for more on this.

In Section~\ref{sec:top} we will discuss another interesting very special case of the spectral sequence in Corollary~\ref{cor:dgass}, namely when the differential graded algebras arise as cochain algebras on spaces.



We are now going to relate the edge homomorphism of the characteristic spectral sequence in Theorem~\ref{thm:main_ss} to the characteristic homomorphism as studied by Buchweitz and Flenner~\cite{Buchweitz+Flenner}, Lowen~\cite{Lowen}, Linckelmann~\cite[Sec.~2]{Linckelmann}, and Kuri\-ba\-ya\-shi~\cite[Sec.~5]{Kuribayashi}, for instance. Since this comparison will involve the more general setting of differential graded categories~$\bfA$ together with a chosen automorphism~$T$, we note that the construction of the spectral sequence in Theorem~\ref{thm:main_ss} extends easily to this more general situation. We have:

\begin{theorem}\label{thm:identification}
The edge homomorphism of the characteristic spectral sequence in Theorem~\ref{thm:main_ss} is the characteristic homomorphism
\begin{equation}\label{eq:char}
\HH_{\dg}^t(\bfA,T)\longrightarrow\rmZ^t_{\gr}(\rmH_\bullet\bfA,\rmH_\bullet T).
\end{equation}
\end{theorem}

\begin{proof}
For the standard coefficient bimodule~$M=\bfA(\Id_\bfA,\Id_\bfA)=\bfA$ as in Example~\ref{ex:der_nat_trf}, the edge homomorphisms of the spectral sequence
\[
\rmE_2^{s,t}\cong\HH_{\gr}^{s,t}(\rmH_\bullet\bfA,\rmH_\bullet T)\Longrightarrow\HH_{\dg}^{s+t}(\bfA,T)
\]
in Theorem~\ref{thm:main_ss} takes the form
\begin{equation}\label{eq:form_of_edge}
\HH_{\dg}^t(\bfA,T)\twoheadrightarrow\rmE_\infty^{0,t}\subseteq\dots\subseteq\rmE_2^{0,t}=\HH_{\gr}^{0,t}(\rmH_\bullet\bfA,\rmH_\bullet T).
\end{equation}
By Proposition~\ref{prop:identification_2}, the right hand side can be identified with the graded center~\hbox{$\rmZ^t_{\gr}(\rmH_\bullet\bfA,\rmH_\bullet T)=\HH_{\gr}^{0,t}(\rmH_\bullet\bfA,\rmH_\bullet T)$} of the homology category as defined in Section~\ref{sec:HM}. This means that we can identify the edge homomorphisms with the homomorphism~\eqref{eq:char} of graded vector spaces.
\end{proof}

\begin{proposition}\label{prop:surjective}
The characteristic homomorphism~\eqref{eq:char} is surjective whenever the spectral sequence in Theorem~\ref{thm:main_ss} degenerates at the~$\rmE_2$ page.
\end{proposition}

\begin{proof}
This is an immediate consequence of~\eqref{eq:form_of_edge}: If the differentials are trivial, then~$\rmE_\infty^{0,t}=\dots=\rmE_2^{0,t}$, and we get
\[
\HH_{\dg}^t(\bfA,T)\twoheadrightarrow\rmE_\infty^{0,t}=\rmE_2^{0,t}=\HH_{\gr}^{0,t}(\rmH_\bullet\bfA,\rmH_\bullet T),
\]
as claimed.
\end{proof}

We will see later in examples that the hypothesis in Proposition~\ref{prop:surjective} is satisfied in substantial classes of examples.


\subsection{The forgetful spectral sequence}\label{sec:forgetful_ss}

Let~$\bfA$ be a differential graded category. The forgetful spectral sequence is the spectral sequence that we obtain from the double complex~\eqref{eq:double} by calculating the differential in the Hochschild direction (that is the~$s$--direction) first. In other words, this boils down to forgetting the differential of~$\bfA$ for a while, treating it simply as a graded category. Thus, the~$\rmE_0$ page is now the structure mentioned in Remark~\ref{rem:not_wanted}, and the~$\rmE_1$ page is its cohomology, the graded Hochschild cohomology groups~$\HH_{\gr}^s(\bfA)$. These are still differential graded vector spaces, so they come with differentials (in the internal~$t$--direction). The cohomology gives rise to the~$\rmE_2$ page of the forgetful spectral sequence. The coordinate transformation~\hbox{$(s,t)=(q,p)$} leads again to a spectral sequence of standard cohomological type, i.e.~such that the differentials on the~$\rmE_r$ page are of the form~\hbox{$\rmE_r^{p,q}\to\rmE_r^{p+r,q-r+1}$}. So we get:

\vbox{\begin{theorem}\label{thm:forgetful_ss}
For every differential graded category~$\bfA$ and~$(\bfA,\bfA)$--bimodule~$M$ there is a spectral sequence of cohomological type with~\hbox{$\rmE_2^{p,q}\cong\rmH^p(\HH_{\gr}^q(\bfA\,;\,M))$} that converges to the Hochschild cohomology~$\HH_{\dg}^{p+q}(\bfA\,;\,M)$ of the differential graded category of~$\bfA$ with coefficients in~$M$.
\end{theorem}}

The spectral sequence in Theorem~\ref{thm:forgetful_ss} will be referred to as the {\it forgetful spectral sequence}. In the special one-object-case we obtain again an interesting particular result:

\begin{corollary}\label{cor:forgetful_ss}
For every differential graded algebra~$C$ and bimodule~$M$ there is a spectral sequence of cohomological type with~\hbox{$\rmE_2^{p,q}\cong\rmH^p(\HH_{\gr}^q(C\,;\,M))$} that converges to the Hochschild cohomology~$\HH_{\dg}^{p+q}(C\,;\,M)$ of~$C$ with coefficients in the bimodule~$M$.
\end{corollary}

\begin{remark}
Weibel~\cite[9.9.1]{Weibel} has some information about a homological analogue of the spectral sequence in Corollary~\ref{cor:forgetful_ss}.
\end{remark}

The forgetful spectral sequence in Corollary~\ref{cor:forgetful_ss} can be used in the following way: Given any differential graded algebra~$C$, or any graded algebra~$C$ with the trivial differential, we can find a differential graded algebra~$F$ together with a quasi-isomorphism~(homology equivalence)~\hbox{$C\leftarrow F$}, and such that~$F$ is free as a graded algebra, i.e.~$F=\rmT V$ is the tensor algebra of some graded vector space~$V$. The first of these properties ensures that the differential graded algebras~$C$ and~$F$ have isomorphic Hochschild cohomologies, compare~\cite[5.3.5]{Loday}. The Hochschild cohomology of the graded algebra~$\rmT V$ (i.e.~ignoring the differential for a while) is easy to compute. There is a resolution
\[
0
\longleftarrow \rmT V
\longleftarrow \rmT V\otimes\rmT V
\longleftarrow \rmT V\otimes V\otimes\rmT V
\longleftarrow 0
\]
of the~$(\rmT V,\rmT V)$--bimodule~$\rmT V$, see~\cite[p.~582]{Loday+Quillen}, for instance. Application of the hom functor~$\Hom(\,?\,,\rmT V)$ in the category of~$(\rmT V,\rmT V)$--bimodules yields the Hochschild cohomology as the cohomology of the complex~\hbox{$\rmT V\longrightarrow\Der(\rmT V,\rmT V)$}, where the differential sends an element~$a$ to its inner derivation~$\rmD_a$. This complex is obviously concentrated in degrees~$q=0$ and~$q=1$, and we get, of course,~\hbox{$\HH^0(\rmT V)=\rmZ(\rmT V)$} and~\hbox{$\HH^1(\rmT V)=\Out(\rmT V)$}, where~$\Out(\rmT V)$ are the outer derivations of the algebra~$\rmT V$, and with differentials inherited from~$\rmT V$ (now remembering them again). All of the higher Hochschild cohomology groups are zero. The forgetful spectral sequence in Corollary~\ref{cor:forgetful_ss} degenerates at~$\rmE_2$ to give short exact sequences
\begin{equation}\label{eq:ses}
0
\longrightarrow\rmH^p(\rmZ(\rmT V))
\longrightarrow\HH^p(C)
\longrightarrow\rmH^{p-1}(\Out(\rmT V))
\longrightarrow0.
\end{equation}
It turns out that the injection on the left hand side of~\ref{eq:ses} is an instance of the forgetful edge homomorphism that we will discuss now and that is always injective when the forgetful spectral sequence degenerates on the~$\rmE_2$ page, see Proposition~\ref{prop:injective} below. 

\begin{remark}
It would be interesting to see if the surjections on the right hand side of~\ref{eq:ses} preserve the natural Lie algebra structures that live on the source and the target: Hochschild cohomology has a Gerstenhaber bracket, and~$\rmH^{p-1}(\Out(\rmT V))$ inherits a Lie algebra structure from the~(outer) derivations.
\end{remark}



Let us consider the forgetful spectral sequence of Theorem~\ref{thm:forgetful_ss} in the case of the standard bimodule~$M=\bfA(\Id_\bfA,\Id_\bfA)=\bfA$ as in Example~\ref{ex:der_nat_trf}, so that it takes the form
\begin{equation}\label{eq:edge}
\rmE_2^{p,q}\cong\rmH^p(\HH_{\gr}^q(\bfA))\Longrightarrow\HH_{\dg}^{p+q}(\bfA).
\end{equation}
In this case, its edge homomorphism is
\[
\rmH^p(\HH_{\gr}^0(\bfA))=\rmE_2^{p,0}\twoheadrightarrow\cdots\twoheadrightarrow\rmE_\infty^{p,0}\subseteq\HH^p_{\dg}(\bfA).
\]

Since~$\HH_{\gr}^0(\bfA)\cong\rmZ_{\gr}^0(\bfA)$ by Proposition~\ref{prop:identification_2}, the left hand side is just the cohomology of the differential graded vector space~$\rmZ_{\gr}^0(\bfA)$, 
so that the edge homomorphism takes the form~\hbox{$\rmH^p(\rmZ_{\gr}(\bfA))\to\HH_{\dg}^p(\bfA)$}. Similarly to Proposition~\ref{prop:surjective} we get:

\begin{proposition}\label{prop:injective}
The edge homomorphism~\hbox{$\rmH^p(\rmZ_{\gr}(\bfA))\to\HH^p_{\dg}(\bfA)$} is an injection whenever the spectral sequence in Theorem~\ref{thm:forgetful_ss} degenerates at the~$\rmE_2$ page.
\end{proposition}

\begin{remark}\label{rem:hlp}
The question if this forgetful edge homomorphism is an isomorphism is one form of the homotopy limit problem (in the sense of Thomason~\cite{Thomason} and Carlsson~\cite{Carlsson}) which in this case compares the cohomology of the strict center with the cohomology of the space of derived natural transformations.
\end{remark}

\begin{remark}\label{rem:no_auto_here}
The discussion in this section extends again easily to the more general situation when an automorphism acts on the differential graded category, as has been the case in Sections~\ref{sec:char_ss}. We can then interprete the discussion above as the special case when the automorphism is given by the parity functor~\eqref{eq:parity}. We have not done so because we will not use this extension in the rest of this paper.
\end{remark}


\section{Modules over the dual numbers}\label{sec:alg}

In this section we consider an extended algebraic example. It concerns a differential graded category~$\bfA=\Per_A$ of perfect chain complexes over an algebra~$A$ which can be thought of as a differential graded algebra concentrated in degree zero. Note that these modules are automatically free when the algebra~$A$ is local. 
The shift-invariant Hochschild cohomology of~$\bfA$ is isomorphic to the usual Hochschild cohomology of the algebra~$A$, or~\hbox{$\HH_{\dg}^\bullet(\Per_A,\Sigma)\cong\HH^\bullet(A)$} in symbols. Proposition~\ref{prop:shift} and its proof demonstrate the necessity of working shift-invariantly. The result is well-known for standard Hochschild cohomology of differential graded categories. It follows from the fact that~$A$ is a generator of~$\Per_A$ in a suitable sense, just as the perfect derived category is the thick subcategory of the derived category generated by~$A$. See Keller's and To\"en's Morita theory in~\cite{Keller:ENS, Keller:ICM, Toen} and the references therein, for instance, or~\cite[Prop.~4.9, Thm.~5.4]{Dwyer+Szymik} for a homotopical version. 



We now choose~$A$ to be the (commutative, local) algebra~\hbox{$A=\mathfrak{K}[\epsilon]/(\epsilon^2)$} of dual numbers, where~$\mathfrak{K}$ is our ground field, and~$\epsilon$ has degree~$0$. It will make a difference whether the characteristic of the ground field is different from~$2$, so that the derivative~$2\epsilon$ of~$\epsilon^2$ with respect to~$\epsilon$ is non-zero, or not.
Let~$\bfA=\Per_{\mathfrak{K}[\epsilon]/(\epsilon^2)}$ be the differential graded category of perfect chain complexes of finitely generated projective~$\mathfrak{K}[\epsilon]/(\epsilon^2)$--modules. Remember that these modules are automatically free since the algebra~$\mathfrak{K}[\epsilon]/(\epsilon^2)$ is local. Examples are the complexes
\begin{equation}\label{ex:dual}
0\longleftarrow \mathfrak{K}[\epsilon]/(\epsilon^2)
\overset{\epsilon}{\longleftarrow} \mathfrak{K}[\epsilon]/(\epsilon^2)
\overset{\epsilon}{\longleftarrow}\cdots
\overset{\epsilon}{\longleftarrow} \mathfrak{K}[\epsilon]/(\epsilon^2)
\longleftarrow 0
\end{equation}
where the modules~$\mathfrak{K}[\epsilon]/(\epsilon^2)$ live in the degrees from~$m$ to~$n$, for some pair of integers~$m\leqslant n$. The category~$\Per_{\mathfrak{K}[\epsilon]/(\epsilon^2)}$ comes with the usual shift-automorphism~$\Sigma$ for differential graded modules.
%
%
%
The Hochschild cohomology of the algebra~$\mathfrak{K}[\epsilon]/(\epsilon^2)$ is well-known. It follows from~\cite{BuenosAires} that the cochain complex~$\CH(\mathfrak{K}[\epsilon]/(\epsilon^2))$ is equivalent to
\[
\mathfrak{K}[\epsilon]/(\epsilon^2)
\overset{0}{\longrightarrow} \mathfrak{K}[\epsilon]/(\epsilon^2)
\overset{2\epsilon}{\longrightarrow} \mathfrak{K}[\epsilon]/(\epsilon^2)
\overset{0}{\longrightarrow}\cdots
\overset{2\epsilon}{\longrightarrow} \mathfrak{K}[\epsilon]/(\epsilon^2)
\longrightarrow\cdots,
\]
so we can read off that
\[
\dim\HH^t(\mathfrak{K}[\epsilon]/(\epsilon^2))=
\begin{cases}
2 & t=0\\
1 & t>0
\end{cases}
\]
in the case when the characteristic of our ground field is different from~$2$. The degree~$0$ part is the (commutative) algebra~$\mathfrak{K}[\epsilon]/(\epsilon^2)$ itself, of course, and the degree~$1$ part is given by the (outer) derivation~$\mathfrak{K}[\epsilon]/(\epsilon^2)\to \mathfrak{K}$ that sends the element~\hbox{$a+\epsilon b$} to~$b$. If the characteristic of our ground field is~$2$, then~$\dim\HH^\bullet(\mathfrak{K}[\epsilon]/(\epsilon^2))=2$ in all non-negative degrees. For the following discussion, it is only important to know that the shift-invariant Hochschild cohomology~$\HH_{\dg}^t(\Per_{\mathfrak{K}[\epsilon]/(\epsilon^2)},\Sigma)$ of the differential graded category~$\Per_{\mathfrak{K}[\epsilon]/(\epsilon^2)}$ is finite-dimensional and non-trivial in all degrees~$t\geqslant0$.
%
%
%
As in Section~\ref{sec:cat}, the homology category~$\rmH_\bullet\Per_{\mathfrak{K}[\epsilon]/(\epsilon^2)}$ of the differential graded category~$\Per_{\mathfrak{K}[\epsilon]/(\epsilon^2)}$ in question is nothing but the perfect derived category~$\Dper_{\mathfrak{K}[\epsilon]/(\epsilon^2)}$ of the algebra~$\mathfrak{K}[\epsilon]/(\epsilon^2)$, the derived category of perfect complexes over the algebra~$\mathfrak{K}[\epsilon]/(\epsilon^2)$, together with the canonical shift automorphism. Its graded center has been described in detail by Krause and Ye~\cite[Sec.~5.4]{Krause+Ye} as a square-zero-extension
\[
\rmZ_{\gr}^\bullet(\Dper_{\mathfrak{K}[\epsilon]/(\epsilon^2)},\Sigma)\cong \mathfrak{K}[\zeta]\oplus(\prod_{d=0}^\infty \mathfrak{K}).
\]
The class~$\zeta$ is of degree~$2$ unless the characteristic of our ground field is~$2$, in which case it is of degree~$1$. We get
\[
\dim\rmZ_{\gr}^t(\Dper_{\mathfrak{K}[\epsilon]/(\epsilon^2)},\Sigma)=
\begin{cases}
\infty & t=0\\
0 & t>0\text{ odd}\\
1 & t>0\text{ even}
\end{cases}
\]
if the characteristic of the ground field is not~$2$, and~\hbox{$\dim\rmZ_{\gr}^t(\Dper_{\mathfrak{K}[\epsilon]/(\epsilon^2)},\Sigma)=1$} in all positive degrees if the characteristic of our ground field is~$2$.
%
%
%
The calculations above allow us to deduce the behavior of the characteristic edge homomorphism in the spectral sequence.

\begin{proposition}
For the differential graded category of perfect complexes over the algebra of dual numbers, the characteristic edge homomorphism is not surjective~(in degree~$0$) and not injective (in positive degrees).
\end{proposition}

\begin{proof}
The graded center of the perfect derived category of the algebra~$\mathfrak{K}[\epsilon]/(\epsilon^2)$ is infinite dimensional in degree~$0$, but the~$0$--th Hochschild cohomology is the finite dimensional algebra~$\mathfrak{K}[\epsilon]/(\epsilon^2)$ itself. It follows that the characteristic edge homomorphism can never be surjective (in degree~$0$). In other words, there are obstructions to lifting classes from the graded center into Hochschild cohomology in the form of differentials supported on these classes. The claim about non-injectivity also follows immediately from the dimension counts above. This means that the spectral sequence contains elements in the spots~$\rmE_2^{s,t}$ for~$s\geqslant1$ that survive the spectral sequence.
\end{proof}


\section{Coherent sheaves over algebraic curves}\label{sec:geo}

In this section we discuss geometric examples where the characteristic homomorphism is not injective or not surjective from the point of view of the spectral sequence in Theorem~\ref{thm:main_ss}.
%
%
Let~$X$ be a smooth algebraic curve over an algebraically closed field. In the smooth case, every coherent sheaf has a finite resolution by a complex of vector bundles. Therefore the differential graded categories of bounded and perfect complexes are equivalent and so the perfect derived category~$\Dper_X$ is the bounded derived category of coherent sheaves over~$X$.
On argues as in the preceding section that the shift-invariant Hochschild cohomology of the differential graded category~$\Per_X$ of perfect complexes is isomorphic to the usual Hochschild cohomology of~$X$, so that we have an isomorphism~\hbox{$\HH_{\dg}^\bullet(\Per_X,\Sigma)\cong\HH^\bullet(X)$}. It will turn out that the behavior of the characteristic homomorphism depends on the genus~\hbox{$\rmg(X)\geqslant0$} of~$X$; it never is an isomorphism. This statement summarizes the following two results.

\begin{proposition}
If the genus~$\rmg(X)$ of the curve~$X$ is non-zero, then the characteristic homomorphism~\hbox{$\HH^2(X)\to\rmZ^2_{\gr}(\Dper_X,\Sigma)$} is not injective.
\end{proposition}

\begin{proof}
Consider the group~$\rmZ^2(\Dper_X,\Sigma)$ on the~$\rmE^2$ page which is a target of the edge homomorphism. In the present case, it must be zero, because smoothness implies that the global dimension of~$X$ is~$1$, and so the category of coherent sheaves on~$X$ is hereditary.
On the other hand, the Hochschild--Kostant--Rosenberg decomposition
\[
\HH^k(X)\cong\bigoplus_{i+j=k}\rmH^i(X\,;\,\Lambda^j\calT_X)
\]
reduces to an isomorphism~$\HH^2(X)\cong\rmH^1(X\,;\,\calT_X)$ in the case of curves. If the genus is~$\rmg(X)=1$, then this is of dimension~$1>0$, and if the genus is~\hbox{$\rmg(X)\geqslant2$}, then this is of dimension~$3\rmg(X)-3>0$. It follows in all cases that the kernel of the edge homomorphism in dimension~$2$ is not zero.
\end{proof}

In other words, there is a non-trivial operation on the differential graded category of coherent complexes on~$X$ that is trivial from the point of view of the derived category. The above argument generalizes a statement of Buchweitz and Flenner \cite[Rem.~3.3.8]{Buchweitz+Flenner} on elliptic curves, who in turn refer to an unpublished preprint~\cite{Caldararu} of C\u{a}ld\u{a}raru. 

\begin{proposition}
If the genus~$\rmg(X)$ of the curve~$X$ is zero, then the characteristic homomorphism~\hbox{$\HH^1(X)\to\rmZ^1_{\gr}(\Dper_X,\Sigma)$} is not surjective.
\end{proposition}

\begin{proof}
Clearly, the algebraic curve~$X$ in question is the projective line~$\bbP^1$. In this case the Hochschild--Kostant--Rosenberg decomposition gives
\[
\dim\HH^t(X)=
\begin{cases}
1 & t=0\\
3 & t=1\\
0 & \text{otherwise}.
\end{cases}
\]

On the other hand, here is a description of the graded center~$\rmZ^\bullet(\Dper_X,\Sigma)$: By Beilinson~\cite{Beilinson}, the category of coherent sheaves on~\hbox{$X=\bbP^1$} is derived equivalent to the category of finite-dimensional representations of an algebra~$A$, namely the Kronecker algebra. This algebra is finite-dimensional. It does not have finite presentation type, but it is tame. The Auslander--Reiten quiver of~$A$ has three types of components: the pre-projective, the pre-injective, and the regular parts. The latter are indexed by the closed points of~$X=\bbP^1$, and they are the only ones that contribute to the graded center by~\cite[Lem.~4.9]{Krause+Ye}. The Auslander--Reiten translation~$\tau$ is the identity on them, so that the situation is~$1$--periodic. We can then use another result of Krause and Ye~\cite[Prop.~4.10]{Krause+Ye} to infer
\[
\dim\rmZ^t(\Dper_X,\Sigma)=
\dim\rmZ^t(\Dper_A,\Sigma)=
\begin{cases}
1 & t=0\\
\infty & t=1\\
0 & \text{otherwise},
\end{cases}
\]
and the edge homomorphism cannot be surjective.
\end{proof}


\section{Free loop spaces and string topology}\label{sec:top} 

In this final section, we will discuss topological examples from the point of view of the spectral sequence in Theorem~\ref{thm:main_ss}. It will turn out that the edge homomorphism~\eqref{eq:char} is rarely injective, but surjective. To put this into context, note that we consider here the special case of differential graded categories~$\bfA$ that have only one object, as in Corollary~\ref{cor:dgass}, and the endomorphism differential graded algebra~$A$ of that object will be the differential graded algebra~$\rmC X$ of cochains~(with coefficients in a ground field~$\mathfrak{K}$) over a topological space~$X$. In particular the omnipresent automorphism is trivial, given by parity. Our grading conventions~$\rmC^p X=(\rmC X)_{-p}$ force these graded algebras to be supported in negative degrees. The characteristic spectral sequence in Theorem~\ref{thm:main_ss} takes the form
\begin{equation}\label{eq:Moore}
\HH^\bullet(\rmH^\bullet X)\Longrightarrow\HH^\bullet(\rmC X)
\end{equation}
in this case. This spectral sequence, also referred to as the {\it Moore spectral sequence}, was constructed directly by Kuribayashi~\cite[Theo.~3.1]{Kuribayashi} in the special case where~$X$ is a connected space with finite type singular cohomology. If~$X$ is a simply-connected space, Cohen and Jones~\cite[Cor.~9]{Cohen+Jones} (see also~\cite{Cohen}) have identified the target with the homology of the free loop space~$\Lambda X$ of~$X$. If~$X=M$ is, in addition, a closed oriented manifold, then this is the string homology of~$M$, and in this special case, a spectral sequence such as~\eqref{eq:Moore} has also been constructed by Felix, Thomas, and Vigu\'e-Poirrier~\cite[3.3~Prop.]{Felixetal}. Our algebraic approach makes it clear that no hypothesis on~$X$ is necessary to obtain this Moore spectral sequence. The edge homomorphisms of the Moore spectral sequence~\eqref{eq:Moore} takes the following form:~\hbox{$\HH^t(\rmC X)\to\rmH^t(X)$}. Note that the target is indeed the entire cohomology ring~$\rmH^\bullet X$: Since it is commutative, it agrees with~$\rmZ^t_{\gr}(\rmH^\bullet X)$.


\begin{figure}
\caption{The spectral sequence~$\HH^\bullet(\rmH^\bullet\rmS^3)\Rightarrow\HH^\bullet(\rmC\rmS^3)$}
\label{fig:S3}
\begin{center}
\begin{tikzpicture}[xscale=1,yscale=.4]
;
\draw[->] (-2,1) -- (4.5,1) node [right=.25] {$t-s$};
\draw (0,1) node [above=.25]{$0$};
\draw (1,1) node [above=.25]{$1$};
\draw (2,1) node [above=.25]{$2$};
\draw (3,1) node [above=.25]{$3$};
\draw (4,1) node [above=.25]{$4$};

\draw[->] (-.5,-13) -- (-.5,4) node [above=.25]{$s$};

\foreach \x in {0,...,4}
	\foreach \y in {-12,...,0}
		\draw (\x,\y) node {$\cdot$};

\draw (-1.25,0) node {$0$};
\draw (0, 0) node {$\bullet$};

\draw (-1.25,-3) node {$-3$};
\draw (0, -3) node {$\bullet$};
\draw (1, -3) node {$\bullet$};

\draw (-1.25,-6) node {$-6$};
\draw (1, -6) node {$\bullet$};
\draw (2, -6) node {$\bullet$};

\draw (-1.25,-9) node {$-9$};
\draw (2, -9) node {$\bullet$};
\draw (3, -9) node {$\bullet$};

\draw (-1.25,-12) node {$-12$};
\draw (3, -12) node {$\bullet$};
\draw (4, -12) node {$\bullet$};

\end{tikzpicture}
\end{center}
\end{figure}

\begin{example}\label{ex:S3}
Figure~\ref{fig:S3} shows part of the Moore spectral sequence~\eqref{eq:Moore} for the~$3$--dimensional sphere~\hbox{$X\simeq\rmS^3$}. The vector spaces
\[
\HH^\bullet(\rmH^\bullet\rmS^3)\cong\rmH^\bullet\rmS^3\otimes\Ext_{\rmH^\bullet\rmS^3}^{\bullet}(\mathfrak{K},\mathfrak{K})
\]
on the~$\rmE_2$ page are displayed as follows. A bullet indicates a~$1$--dimensional vector space; a dot indicates that the vector space is trivial. The more general case when~$X\simeq\rmS^n$ for an odd integer~\hbox{$n\geqslant5$} is similar. Compare with Kuribayahsi's computation~\cite[Prop.~2.4]{Kuribayashi}. Sparseness implies that these spectral sequences degenerate. We see that the characteristic edge homomorphisms are surjective, but they have a non-trivial kernel. 
\end{example}


The following Proposition~\ref{prop:formal} demonstrates that the behavior in the preceding Example~\ref{ex:S3} is typical. Recall that a topological space is {\em formal}~(over~$\mathfrak{K}$) if there are homomorphisms~\hbox{$\rmC X\leftarrow A \rightarrow \rmH^\bullet(X)$} of differential graded algebras inducing isomorphisms in cohomology. Here~$\rmH^\bullet(X)$ is the differential graded algebra with trivial differential~$\delta=0$.
See~\cite{Deligneetal},~\cite{Sullivan}, or~\cite{FOT} for many interesting geometric and topological examples of formality.


\begin{proposition}\label{prop:formal}
Let~$X$ be a formal space. Then the characteristic edge homomorphism~$\HH^t(\rmC X)\to\rmH^t(X)$ is an epimorphism for all~$t$.
\end{proposition}

\begin{proof}
If~$X$ is a formal space, then we know that there exists a sequence of quasi-isomorphisms connecting the differential graded cochain algebra~$\rmC X$ with the discrete cohomology algebra~$\rmH^\bullet X$. From the invariance of Hochschild cohomology under quasi-isomorphisms of differential graded algebras~(as proved in~\cite[5.3.5]{Loday} and~\cite[Sec.~3]{FMT}) it follows that there exists an isomorphism~\hbox{$\HH^\bullet(\rmC X)\cong \HH^\bullet(\rmH^\bullet X)$} of algebras. Therefore the Moore spectral sequence~\eqref{eq:Moore} degenerates and the characteristic edge homomorphism
\[
\HH^t(\rmC X)\twoheadrightarrow\rmE_\infty^{0,t}\cong\dots\cong\rmE_2^{0,t}=\HH^{0,t}(\rmH^\bullet X)\cong\rmZ^t_{\gr}(\rmH^\bullet X).
\]
is an epimorphism for all~$t$.
\end{proof}

\begin{example}\label{ex:Heisenberg}
Let~$\mathfrak{K}$ be a field of characteristic~$0$. Let us consider an example of a manifold that is not formal~(over the field~$\mathfrak{K}$) and that is not even simply-connected, namely the~$3$--dimensional Heisenberg manifold~$X$. (See~\cite[Ex.~3.17]{FOT} for an introduction.) The Heisenberg manifold is actually an aspherical manifold, namely we have an equivalence~\hbox{$X\simeq\rmB G$}, where~$\rmB G$ is the classifying space of the integral Heisenberg group~$G$ of unit triangular~$(3,3)$--matrices. The cohomology algebra is~$6$--dimensional, namely~\hbox{$\rmH^\bullet X\cong(\Lambda(a_1,a_2)\otimes\Lambda(b_1,b_2))/(a_1a_2=b_1b_2)$}, where the subscripts indicate the degree. Since~$G$ is nilpotent but not abelian, we know from general principles that~$X$ cannot be formal~\cite{Hasegawa}. This can also be seen directly: The Massey product~$\langle a_1,a_1,b_1\rangle$ is~$a_2\not=0$. The Hochschild cohomology~$\HH^\bullet(\rmH^\bullet X)$ contains~$\HH^0(\rmH^\bullet X)=\rmH^\bullet X$, since the ordinary cohomology algebra of a space is always commutative. This is the zeroth column of the Moore spectral sequence. As for the target~$\HH^\bullet(\rmC X)$ of that spectral sequence, note that the cochain algebra~$\rmC X\simeq \rmC\rmB G$ has finite-dimensional homology. It is the Koszul dual of the group algebra~$\mathfrak{K}[G]$ over the ground field~$\mathfrak{K}$: Since the Koszul dual of any differential graded algebra is the linear dual of the bar construction, we get~$\mathfrak{K}[G]^!=\rmC\rmB G$ in our case. It follows from the invariance of Hochschild cohomology under Koszul duality~\cite[Thm.~4.1]{Hu} that there exists an isomorphism~\hbox{$\HH^\bullet(\rmC X)\cong\HH^\bullet(\mathfrak{K}[G])$}, and the right hand side can be computed as the product of the group cohomologies of the centralizers of the elements in~$G$. The centralizer of the unit is~$G$ itself, so that~$\rmH^\bullet\rmB G\cong\rmH^\bullet X$ also appears in the target of the spectral sequence, and the edge homomorphism is surjective.
\end{example}

More generally, we have the following result.

\begin{proposition}\label{prop:Hodge} 
Let~$\mathfrak{K}$ be a field of characteristic~$0$ and~$X$ be any topological space. 
Then the characteristic edge homomorphism~$\HH^t(\rmC X)\to\rmH^t(X)$ is surjective.
\end{proposition}

\begin{proof}
In characteristic~$0$, the~$\rmE_\infty$ differential graded algebra~$\rmC X$ of cochains on the space~$X$ can be replaced by a quasi-isomorphic differential graded algebra~$C$ that is~(strictly) commutative~(in the graded sense), so that we have isomorphisms~$\HH^t(C)\cong\HH^t(\rmC X)$ and~$\rmH^\bullet(C)\cong\rmH^\bullet(\rmC X)\cong\rmH^\bullet(X)$. We can then use the Hodge decomposition~(Loday~\cite[Sec.~5.4.8]{Loday} or Weibel~\cite[Sec.~9.4.3]{Weibel}) for the Hochschild cohomology of commutative algebras in characteristic~$0$, where~$\rmH^\bullet(X)$ appears as the~$0$--th summand, to see that the cohomology has to survive the spectral sequence.
\end{proof}


Since the first version of this paper appeared, Briggs and G\'elinas published a preprint~\cite{Briggs+Gelinas} where they discuss the characteristic homomorphism in the~$\rmA_\infty$ setting, and the full~$\rmE_\infty$ structure might be useful to extend the above result to positive characteristics.


\section*{Acknowledgments}

The first author thanks the Department of Mathematical Sciences of NTNU for the wonderful hospitality and financial support. He likes to thank Teimuraz Pirashvili and Nicole Snashall for many inspiring conversations. The second author has been supported by the Danish National Research Foundation through the Centre for Symmetry and Deformation (DNRF92). He thanks Petter Bergh, Reiner Hermann, and \O yvind Solberg. Both authors thank Benjamin Briggs, Vincent G\'elinas, and the referee for their valuable comments and suggestions.



\parbox{\linewidth}{%
Frank Neumann,
Department of Mathematics, 
University of Leicester,
Leicester LE1 7RH,
UNITED KINGDOM\phantom{ }\\
\href{mailto:fn8@le.ac.uk}{fn8@le.ac.uk}}

\vspace{\baselineskip}

\parbox{\linewidth}{%
Markus Szymik,
Department of Mathematical Sciences,
NTNU Norwegian University of Science and Technology,
7491 Trondheim,
NORWAY\phantom{ }\\
\href{mailto:markus.szymik@ntnu.no}{markus.szymik@ntnu.no}}


\begin{thebibliography}{GGRSV91}

\bibitem[AI07]{Avramov+Iyengar} L.L. Avramov, S.B. Iyengar. Constructing modules with prescribed cohomological support. Illinois J. Math. 51 (2007) 1--20.

\bibitem[BW85]{Baues+Wirsching} H.J. Baues, G. Wirsching. Cohomology of small categories. J. Pure Appl. Algebra 38 (1985) 187--211.

\bibitem[Bei78]{Beilinson} A.A. Be\u{\i}linson. Coherent sheaves on~$\bbP^n$ and problems in linear algebra. Funct. Anal. Appl. 12 (1978) 214--216.

\bibitem[Ber84]{Bernstein} J.-N. Bernstein. Le ``centre'' de Bernstein. 
R\'edig\'e par P. Deligne. Travaux en Cours. Representations of reductive groups over a local field, 1--32. Hermann, Paris, 1984.

\bibitem[SGA6]{SGA6} P. Berthelot, A. Grothendieck, L. Illusie. S\'eminaire de G\'eom\'etrie Alg\'ebrique du Bois Marie - 1966--67 - Th\'eorie des intersections et th\'eor\`eme de Riemann--Roch - SGA 6. Lecture Notes in Mathematics 225. Berlin, New York, Springer-Verlag, 1971.

\bibitem[BG]{Briggs+Gelinas} B. Briggs, V. G\'elinas. The $\rmA_\infty$--centre of the Yoneda algebra and the characteristic action of Hochschild cohomology on the derived category. Preprint. \href{https://arxiv.org/abs/1702.00721}{arxiv.org/abs/1702.00721}

\bibitem[BF08]{Buchweitz+Flenner} R.-O. Buchweitz, H. Flenner. Global {Hochschild} (co-)homology of singular spaces. Adv. Math. 217 (2008) 205--242.

\bibitem[C\u{a}l]{Caldararu} A. C\u{a}ld\u{a}raru. The Mukai pairing, I: the Hochschild structure. Pre\-print. \href{http://arxiv.org/abs/math/0308079}{arxiv.org/abs/math/0308079}

\bibitem[Car87]{Carlsson} G. Carlsson. Segal's Burnside ring conjecture and the homotopy limit problem. Homotopy theory (Durham, 1985) 6--34. London Math. Soc. Lecture Note Ser. 117. Cambridge Univ. Press, Cambridge, 1987.

\bibitem[GGRSV91]{BuenosAires} J.A. Guccione, J.J. Guccione, M.J. Redondo, A. Solotar, O.E. Villamayor (Buenos Aires Cyclic Homology Group). Cyclic homology of algebras with one generator. K-Theory 5 (1991) 51--69.

\bibitem[Coh04]{Cohen} R.L. Cohen. Multiplicative properties of Atiyah duality. Homology Homotopy Appl. 6 (2004) 269--281.

\bibitem[CJ02]{Cohen+Jones} R.L. Cohen, J.D.S. Jones. A homotopy theoretic realization of string topology. Math. Ann. 324 (2002) 773--798.

\bibitem[DGMS75]{Deligneetal} P. Deligne, P. Griffiths, J. Morgan, D. Sullivan, Real homotopy theory of K\"ahler manifolds. Invent. Math. 39 (1975) 245--274.

\bibitem[DS17]{Dwyer+Szymik} W.G. Dwyer, M. Szymik. Frobenius and the derived centers of algebraic theories. Math. Z. 285 (2017) 1181--1203.

\bibitem[EKMM97]{EKMM} A.D. Elmendorf, I. Kriz, M.A. Mandell, J.P. May. Rings, modules, and algebras in stable homotopy theory. Mathematical Surveys and Monographs 47. American Mathematical Society, Providence, RI, 1997.

\bibitem[ENO05]{ENO} P. Etingof, D. Nikshych, V. Ostrik. On fusion categories. Ann. of Math. 162 (2005) 581--642.


\bibitem[FMT05]{FMT} Y. Felix, L. Menichi, J.-C. Thomas. Gerstenhaber duality in Hoch\-schild cohomology. J. Pure Appl. Algebra 199 (2005) 43--59.

\bibitem[FOT08]{FOT} Y. Felix, J. Oprea, D. Tanr\'e. Algebraic Models in Geometry. Oxford Grad. Texts in Math. 17. Oxford University Press, 2008.

\bibitem[FTV-P04]{Felixetal} Y. Felix, J.-C. Thomas, M. Vigu\'e-Poirrier. The Hochschild cohomology of a closed manifold. 
Inst. Hautes \'Etudes Sci. Publ. Math. 99 (2004) 235--252.


\bibitem[Has89]{Hasegawa} K. Hasegawa. Minimal models of nilmanifolds. Proc. Amer. Math. Soc. 106 (1989) 65--71.

\bibitem[HPS97]{HPS} M. Hovey, J.H. Palmieri, N.P. Strickland. Axiomatic stable homotopy theory. Mem. Amer. Math. Soc. 128 (1997) 610. 

\bibitem[Hu06]{Hu} P. Hu. Higher string topology on general spaces. Proc. London Math. Soc. 93 (2006) 515--544. 

\bibitem[Kel94]{Keller:ENS} B. Keller. Deriving DG categories. Ann. Sci. \'Ecole Norm. Sup. 27 (1994) 63--102. 

\bibitem[Kel06]{Keller:ICM} B. Keller. On differential graded categories. International Congress of Mathematicians. Vol. II, 151--190. Eur. Math. Soc., Z\"urich, 2006.

\bibitem[Kel65]{Kelly} G.M. Kelly. Chain maps inducing zero homology maps. Proc. Cambridge Philos. Soc. 61 (1965) 847--854. 


\bibitem[KY11]{Krause+Ye} H. Krause, Y. Ye. On the centre of a triangulated category. Proc. Edinb. Math. Soc. 54 (2011) 443--466. 

\bibitem[Kur11]{Kuribayashi} K. Kuribayashi. The Hochschild cohomology ring of the singular cochain algebra of a space. Ann. Inst. Fourier (Grenoble) 61 (2011) 1779--1805.

\bibitem[Lin09]{Linckelmann} M. Linckelmann. On graded centres and block cohomology. Proc. Edinb. Math. Soc. 52 (2009) 489--514.

\bibitem[Lod92]{Loday} J.-L. Loday. Cyclic homology. Grundlehren der Mathematischen Wissenschaften 301. Springer-Verlag, Berlin, 1992.

\bibitem[LQ84]{Loday+Quillen} J.-L. Loday, D. Quillen. Cyclic homology and the Lie algebra homology of matrices. Comment. Math. Helv. 59 (1984) 565--591. 

\bibitem[Low08]{Lowen} W. Lowen. Hochschild cohomology, the characteristic morphism and derived deformations. Compos. Math. 144 (2008) 1557--1580.

\bibitem[LvdB05]{Lowen+vandenBergh} W. Lowen, M. van den Bergh. {Hochschild} cohomology of abelian categories and ringed spaces. Adv. Math. 198 (2005) 172--221.

\bibitem[MS15]{Meir+Szymik} E. Meir, M. Szymik. Drinfeld centers for bicategories. Doc. Math. 20 (2015) 707--735.

\bibitem[Mit72]{Mitchell} B. Mitchell. Rings with several objects. Adv. Math. 8 (1972) 1--161.


\bibitem[Shi07]{Shipley} B. Shipley.~$\rmH\bbZ$--algebra spectra are differential graded algebras. Amer. J. Math. 129 (2007) 351--379.

\bibitem[Sul77]{Sullivan} D. Sullivan. Infinitesimal computations in topology. Inst. Hautes \'Etudes Sci. Publ. Math. 46 (1977) 269--331. 

\bibitem[Szy]{Szymik} M. Szymik. Homotopy coherent centers versus centers of homotopy categories. Preprint. \href{http://arxiv.org/abs/1305.3029}{arxiv.org/abs/1305.3029}

\bibitem[Tho83]{Thomason} R.W. Thomason. The homotopy limit problem. Proceedings of the Northwestern Homotopy Theory Conference (Evanston, Ill., 1982) 407--419. Contemp. Math. 19. Amer. Math. Soc., Providence, R.I., 1983.

\bibitem[To{\"e}07]{Toen} B. To\"en. The homotopy theory of dg-categories and derived Morita theory. Invent. Math. 167 (2007) 615--667. 

\bibitem[TV13]{TV} V. Turaev, A. Virelizier. On the graded center of graded categories. J. Pure Appl. Algebra 217 (2013) 1895--1941. 

\bibitem[Wei94]{Weibel} C.A. Weibel. An introduction to homological algebra. Cambridge University Press, Cambridge, 1994. 

\end{thebibliography}
\end{document}